\documentclass[final,3p]{elsarticle}

\usepackage{amssymb}
\usepackage{amsmath}
\usepackage{amsthm}
\usepackage{booktabs}
\usepackage{multicol}
\usepackage{multirow}
\usepackage{makecell}
\usepackage{algorithm, algorithmic}

\newcommand{\vseq}[2]{%
  \rotatebox[origin=c]{90}{%
    \begin{tabular}{@{}c@{}}
      \textbf{#1}\\[-1pt]
      {\scriptsize #2}
    \end{tabular}%
  }%
}

\usepackage[hidelinks]{hyperref}

\theoremstyle{plain}
\newtheorem{theorem}{Theorem}

\newtheorem*{lemma*}{Lemma}

\theoremstyle{definition}     
\newtheorem{definition}{Definition}

\newtheorem{remark}{Remark}

\numberwithin{theorem}{section}
\numberwithin{definition}{section}
\numberwithin{lemma}{section}
\numberwithin{proposition}{section}
\numberwithin{corollary}{section}
\numberwithin{notation}{section}
\numberwithin{remark}{section}
\numberwithin{example}{section}

\begin{document}

\begin{frontmatter}

\title{Tensor CUR Decomposition under the Linear-Map-Based Tensor-Tensor Multiplication}

\author[label1,label2,label3]{Susana López-Moreno}
\author[label1,label2,label3]{June-Ho Lee}
\author[label4]{Taehyeong Kim\corref{cor1}}
\ead{thkim0519@knu.ac.kr}

\affiliation[label1]{organization={Department of Mathematics, Pusan National University},
            city={Busan},
            country={Republic of Korea}}

\affiliation[label2]{organization={Humanoid Olfactory Display Center, Pusan National University},
            city={Yangsan-si},
            state={Gyeongsangnam-do},
            country={Republic of Korea}}
            
\affiliation[label3]{organization={Industrial Mathematics Center, Pusan National University},
            city={Busan},
            country={Republic of Korea}}
            
\affiliation[label4]{organization={Nonlinear Dynamics and Mathematical Application Center, Kyungpook National University},
            city={Daegu},
            country={Republic of Korea}}
            
\cortext[cor1]{Corresponding author}

\begin{abstract}
\small
The factorization of three-dimensional data continues to gain attention due to its relevance in representing and compressing large-scale datasets.
The linear-map-based tensor–tensor multiplication is a matrix-mimetic operation that extends the notion of matrix multiplication to higher-order tensors, and which is a generalization of the T-product.
Under this framework, we introduce the tensor CUR decomposition, show its performance in video foreground-background separation for different linear maps and compare it to a robust matrix CUR decomposition, another tensor approximation and the slice-based singular value decomposition (SS-SVD). We also provide a theoretical analysis of our tensor CUR decomposition, extending classical matrix results to establish exactness conditions and perturbation bounds.
\end{abstract}


\end{frontmatter}

\section{Introduction}
\label{sec:introduction}
Low rank tensor decomposition remains an active research area, driven by the fact that multiway structure arises naturally in data analysis, including video and imaging, signal processing, and machine learning, and exploiting that structure can yield models that are both more compact and faithful than those obtained by flattening the data.

A notable recent direction is matrix-mimetic tensor algebra, where tensors are treated as linear operators under an algebraic product that is related to matrix multiplication.
The linear-map-based tensor-tensor multiplication was introduced in \citep{kernfeld2015tensor} as a generalization of the classical T-product \citep{kilmer2011factorization} and as a way of factorizing a tensor as the product of two tensors without using the block matrix structure of the latter.
Under this multiplication, a tensor singular value decomposition (SVD) was also defined. It was subsequently extended for orthogonal surjective linear maps in \citep{keegan2025projected} and for injective linear maps in \citep{ju2025computation}.
For orthogonal surjective maps, the corresponding tensor SVD theory and experiments show that its SVD can outperform higher-order singular value decomposition (HOSVD) approximations on multiway data such as hyperspectral images. Despite these advances, few tensor factorizations beyond the SVD family have been developed within the linear-map-based tensor-tensor multiplication.
This motivates the extension of other classical matrix decompositions that are central in data analysis such as CUR decompositions, which has been done for other tensor-tensor multiplications such as the T-product in \citep{chen2022tensor}.

The remainder of the paper is organized as follows. In Section~\ref{sec:preliminaries}, we introduce concepts and notation needed to understand the paper.
In Section~\ref{sec:definition_CUR}, we present the proposed tensor CUR, followed by its theoretical analysis in Section~\ref{sec:theoretical_analysis}.
Finally, in Section~\ref{Sec:numerical_experiments}, we formulate the problem under consideration and provide experimental results on foreground--background separation in video data and compare its performance with a robust matrix-wise CUR decomposition \citep{cai2021robust}, a tensor Bhattacharya–Messner decomposition (BM) \citep{tian2025tensor}, and the spatiotemporal slice-based singular value decomposition (SS-SVD) \citep{kajo2018svd}.

\section{Preliminaries}
\label{sec:preliminaries}
\subsection{Notation}
We establish the notation that will be used throughout the paper. Let $\mathbb{K}=\mathbb{R},\mathbb{C}$.
We consider tensors of order at most three: vectors of size $m$,
matrices of size $m\times n$, and tensors of size $m\times n\times p$.
The linear-map-based tensor-tensor multiplication regards third-order tensors $\mathcal{A}\in\mathbb{K}^{m\times n\times p}$ as stacks of $p$ matrices of size $m\times n$, called the \textit{frontal slices} and denoted by $\mathcal{A}^{(i)}$, $i\in\{1,\dots,p\}$.
We express a tensor as $\mathcal{A}=\left[\begin{array}{c|c|c|c}
     \mathcal{A}^{(1)}& \mathcal{A}^{(2)}&\cdots&\mathcal{A}^{(p)}
\end{array}\right]$.
MATLAB-like indexing will be used throughout: for a matrix $M$, $M(I,:)$ and $M(:,J)$ denote submatrices formed by the rows and columns indexed by $I$ and
$J$, respectively.
Lateral slices of a tensor are denoted as $\mathcal{A}(:,J,:) \in \mathbb{K}^{m \times |J| \times p}$, with $J\subset\{1,\dots,n\}$.
The mode-3 product of a tensor $\mathcal{A}$ and a matrix $M$ is denoted by $\widehat{\mathcal{A}}=\mathcal{A}\times_3 M$, and the symbol $\triangle$ denotes the facewise  product of two tensors.
\subsection{Linear-map-based tensor-tensor multiplication}

\begin{definition}
Let $M \in \mathbb{K}^{p \times p}$ be invertible, and let $\mathcal{A}\in\mathbb{K}^{m\times n \times p}$ and $\mathcal{B}\in\mathbb{K}^{n\times s \times p}$.
The \emph{tensor-tensor multiplication} (or \emph{$*_M$-product}) of $\mathcal{A}$ and $\mathcal{B}$ is the tensor $\mathcal{A}*_M\mathcal{B}\in\mathbb{K}^{m\times s\times p}$ defined by
\begin{equation*}
\mathcal{A}*_M\mathcal{B}=((\mathcal{A} \times_3 M)\triangle (\mathcal{B} \times_3 M))\times_3 M^{-1}.
\end{equation*}
For surjective (respectively, injective) matrices, the inverse is replaced by the pseudoinverse $M^+=M^H(MM^H)^{-1}$ (respectively, $M^+=(M^HM)^{-1}M^H$), which yields $\mathcal{A}*_M\mathcal{B}=((\mathcal{A} \times_3 M)\triangle (\mathcal{B} \times_3 M))\times_3 M^{+}$.
\end{definition}

For the theoretical analysis in Section \ref{sec:theoretical_analysis} of the proposed tensor CUR decomposition, we introduce the following notion of multirank under the $*_M$-product.
\begin{definition}\label{defn:multirank}
Let $\mathcal{A}\in\mathbb{K}^{m\times n\times p}$ and let $M\in\mathbb{K}^{q\times p}$ be of full-rank.
The \textit{multirank} of $\mathcal{A}$ with respect to the $*_M$-product is defined as
\begin{equation*}
    \operatorname{rank}_m(\mathcal{A})=(\operatorname{rank}(\widehat{\mathcal{A}}^{(1)}),\dots,\operatorname{rank}(\widehat{\mathcal{A}}^{(q)})),
\end{equation*}
where $\widehat{\mathcal{A}}^{(k)}$ denotes the $k$-th frontal slice of $\widehat{\mathcal{A}}$.
\end{definition}

\begin{definition}
Let $\mathcal{A}\in\mathbb{K}^{m\times n\times p}$ and let $M\in\mathbb{K}^{q\times p}$ be of full-rank.
The \textit{tensor spectral norm with respect to the $*_M$-product} is defined as
\begin{equation*}
    \| \mathcal{A} \|_{2,*_M} =\max_{k=1,\dots,q} \|
\widehat{\mathcal{A}}^{(k)}\|_2,
\end{equation*}
where $\|\cdot\|_2$ denotes the spectral norm in the matrix sense.
\end{definition}

\section{Tensor CUR decomposition under the linear-map-based tensor-tensor multiplication}\label{sec:definition_CUR}

Given $\mathcal{A}\in\mathbb{K}^{m\times n\times p}$ and $M\in\mathbb{K}^{q\times p}$ of full rank, a practical construction of the $*_M$-CUR decomposition would proceed as follows: Fix column and row index sets $J\subset\{1,\dots,n\}$ and $I\subset\{1,\dots,m\}$ and get $\mathcal{C}:=\mathcal{A}(:,J,:)$,
$\mathcal{R}:=\mathcal{A}(I,:,:)$, $\mathcal{U}:=\mathcal{A}(I,J,:)$.
Map $\mathcal{A}$ to $\widehat{\mathcal{A}}:=\mathcal{A}\times_3 M$, which yields frontal slices $\widehat{\mathcal{A}}^{(k)}\in\mathbb{K}^{m\times n}$ for $k=1,\dots,q$.
Similarly map $\mathcal{C}$, $\mathcal{R}$ and $\mathcal{U}$ to $\widehat{\mathcal{C}}$, $\widehat{\mathcal{R}}$ and $\widehat{\mathcal{U}}$, respectively.
On each slice perform the matrix CUR decomposition
\begin{equation*}
    \widehat{\mathcal{A}}^{(k)} = \widehat{\mathcal{C}}^{(k)}(\widehat{\mathcal{U}}^{(k)})^+\widehat{\mathcal{R}}^{(k)} + \mathcal{E}^{(k)},
\end{equation*}
where $\widehat{\mathcal{C}}^{(k)}:=\widehat{\mathcal{A}}^{(k)}(:,J)$, $\widehat{\mathcal{R}}^{(k)}:=\widehat{\mathcal{A}}^{(k)}(I,:)$, $(\widehat{\mathcal{U}}^{(k)})^+$ is the pseudoinverse of $\widehat{\mathcal{U}}^{(k)}$ for each $k$, and where $\mathcal{E}$ is the slice-wise residual error tensor.
Stack $\widehat{\mathcal{C}}^{(k)},\widehat{\mathcal{U}}^{(k)},\widehat{\mathcal{R}}^{(k)}$ as frontal slices of $\widehat{\mathcal{C}}:=\widehat{\mathcal{A}}(:,J,:)$, $\widehat{\mathcal{U}}\in\mathbb{K}^{|I|\times |J|\times q}$, and $\widehat{\mathcal{R}}:=\widehat{\mathcal{A}}(I,:,:)$.
In the original domain define $\mathcal{U}^{+}:=\widehat{\mathcal{U}}^{+}\times_3 M^{+}$,  where $\widehat{\mathcal{U}}^+$ refers to the tensor $\widehat{\mathcal{U}}^+:=\left[\begin{array}{c|c|c}
     (\widehat{\mathcal{U}}^{(1)})^+ & \cdots & (\widehat{\mathcal{U}}^{(q)})^+
\end{array}\right]$.
The $*_M$-CUR approximation of $\mathcal{A}$ is, then,
\begin{equation*}
\mathcal{A}\approx\mathcal{C}*_M \mathcal{U}^+*_M \mathcal{R}.
\end{equation*}

\begin{remark}
Depending on the properties of the linear map $M$, the resulting decomposition exhibits different properties. We summarize the distinctions between the invertible, surjective, and injective cases in Table \ref{tab:M_cases}. While our theoretical analysis primarily focuses on the invertible case for rigor, the numerical experiments in Section \ref{Sec:numerical_experiments} cover all three scenarios.
\end{remark}

\begin{table}[h]
\centering
\caption{Characteristics of $*_M$-CUR decomposition based on the linear map $M$.}
\label{tab:M_cases}
\small
\begin{tabular}{l|l|l}
\toprule
\textbf{Case} & \textbf{Condition on $M \in \mathbb{K}^{q \times p}$} & \textbf{Decomposition Characteristics} \\
\midrule
Invertible & $q=p$, full Rank & Possible exact reconstruction: $\mathcal{A} = \mathcal{C} *_M \mathcal{U}^+ *_M \mathcal{R}$. \\
\midrule
Surjective & $q < p$, full Rank & Approximation on projected space: $\mathcal{A} \times_3 (M^+ M)$. \\
\midrule
Injective & $q > p$, full Rank &  \makecell[l]{Non-associative product $\Rightarrow$ $*_M$-pseudo-CUR.\\ Maps $\mathcal{A}$ to $\widehat{\mathcal{A}} \times_3 M^+$ (See \cite{ju2025computation}).} \\
\bottomrule
\end{tabular}
\end{table}

\subsection{Theoretical analysis}\label{sec:theoretical_analysis}
Due to the construction of the $*_M$-product, the equality of the tensor CUR decomposition is a result derived from the matrix case.

\begin{theorem}\label{thm:equality_CUR}
Let $\mathcal{A}\in\mathbb{K}^{m\times n\times p}$ and let $M\in\mathbb{K}^{p\times p}$ be invertible.
Fix index sets $I\subset\{1,\dots,m\}$, $J\subset\{1,\dots,n\}$ such that
$\operatorname{rank}_m(\mathcal{U})=\operatorname{rank}_m(\mathcal{A})$.
Then, the $*_M$-CUR reconstruction is exact, i.e.,
\begin{equation}\label{eq:exact_CUR}
    \mathcal{A}=\mathcal{C}*_M\mathcal{U}^+*_M\mathcal{R}.
\end{equation}
\end{theorem}

\begin{proof}
Let $\mathcal{A}$ be a tensor in $\mathbb{K}^{m\times n\times p}$ and let $\widehat{\mathcal{A}}$ be its image under the mode-3 product with $M\in\mathbb{K}^{p\times p}$.
Under the rank condition
$\operatorname{rank}(\widehat{\mathcal{U}}^{(k)})=\operatorname{rank}(\widehat{\mathcal{A}}^{(k)})$ for all $k=\{1,\dots,p\}$,
the matrix CUR exactness identity, proven in \citep{cai2020rapid}, yields
\begin{equation*}
\widehat{\mathcal{A}}^{(k)} = \widehat{\mathcal{C}}^{(k)}\big(\widehat{\mathcal{U}}^{(k)}\big)^+ \widehat{\mathcal{R}}^{(k)}
\end{equation*}
for each $k$.
By stacking the matrices $\widehat{\mathcal{C}}=\left[\begin{array}{c|c|c}
     \widehat{\mathcal{C}}^{(1)}&\cdots&\widehat{\mathcal{C}}^{(p)}
\end{array}\right]$, $\widehat{\mathcal{U}}=\left[\begin{array}{c|c|c}
     \widehat{\mathcal{U}}^{(1)}&\cdots&\widehat{\mathcal{U}}^{(p)}
\end{array}\right]$, $\widehat{\mathcal{R}}=\left[\begin{array}{c|c|c}
     \widehat{\mathcal{R}}^{(1)}&\cdots&\widehat{\mathcal{R}}^{(p)}
\end{array}\right]$, and by performing the facewise product, we get that
\begin{equation*}
\widehat{\mathcal{A}}= \widehat{\mathcal{C}}\triangle(\widehat{\mathcal{U}})^+\triangle \widehat{\mathcal{R}}.
\end{equation*}

We then have that by the properties of the mode-$3$ and the $*_M$-product for invertible $M$,
\begin{equation*}
\begin{aligned}
\mathcal{A} = \widehat{\mathcal{A}}\times_3M^+
&=\begin{bmatrix}(\widehat{\mathcal{C}}\triangle\widehat{\mathcal{U}}^+)\triangle \widehat{\mathcal{R}}\end{bmatrix}\times_3 M^+ \\
&=\begin{bmatrix} \big(\begin{bmatrix}(\mathcal{C}\times_3 M)\triangle (\mathcal{U}^+\times_3M)\end{bmatrix}\times_3 (MM^+)\big)\triangle (\mathcal{R}\times_3 M)\end{bmatrix}\times_3 M^+\\
&= \begin{bmatrix}\big((\mathcal{C}*_M\mathcal{U^+})\times_3 M\big)\triangle (\mathcal{R}\times_3 M)\end{bmatrix}\times_3 M^+ \\
&=\mathcal{C}*_M\mathcal{U}^+*_M\mathcal{R},
\end{aligned}
\vspace{-8pt}
\end{equation*}
where $\mathcal{U}^+=(\widehat{\mathcal{U}})^+\times_3 M^+$.
\end{proof}

Perturbation analysis on the $*_M$-CUR decomposition can also be generalized from the matrix case by applying it slice-wise on the domain induced by  $\times_3 M$.

\begin{theorem}\label{thm:perturbation_CUR}
Let $M\in\mathbb{K}^{p\times p}$ be invertible and let $\mathcal{A}\in\mathbb{K}^{m\times n\times p}$.
Let $\mathcal{A}_{\mathcal{E}}= \mathcal{A} + \mathcal{E}$, such that $\|
\mathcal{E}\|_{2,*_M}$ is sufficiently small, and construct $\mathcal{C}_{\mathcal{E}}=\mathcal{A}_{\mathcal{E}}(:,J,:)$, $\mathcal{R}_{\mathcal{E}}=\mathcal{A}_{\mathcal{E}}(I,:,:)$, and $\mathcal{U}_{\mathcal{E}}=\mathcal{A}_{\mathcal{E}}(I,J,:)$.
Define the $*_M$-CUR approximation of $\mathcal{A}_{\mathcal{E}}$ as
$\widetilde{\mathcal{A}}=\mathcal{C}_{\mathcal{E}}*_M\mathcal{U}_{\mathcal{E}}^+*_M \mathcal{R}_{\mathcal{E}}$, and assume $\sigma_r(\widehat{\mathcal{U}}^{(k)})>2\mu \|
\widehat{\mathcal{E}}^{(k)}(I,J)\|_2$, with $\mu\in [1,3]$ as in \cite[Theorem 4.6]{hamm2021perturbations}. Then, there exists a constant $C$ depending only on $\|\mathcal{C}\|_{2,*_M}$, $\|\mathcal{U}^+\|_{2,*_M}$, and $\|\mathcal{R}\|_{2,*_M}$ such that
\begin{equation*}
    \|
\mathcal{A} - \widetilde{\mathcal{A}}\|_{2,*_M}\leq C \| \mathcal{E}\|_{2,*_M}.
\end{equation*}
\end{theorem}

\begin{proof}
Let $\widehat{\mathcal{A}}=\mathcal{A}\times_3M$ and $\widehat{\mathcal{A}}_{\mathcal{E}}=\mathcal{A}_{\mathcal{E}}\times_3 M$. Then $\widehat{\mathcal{A}}_{\mathcal{E}}=\widehat{\mathcal{A}}+\widehat{\mathcal{E}}$, where $\widehat{\mathcal{E}}:=\mathcal{E}\times_3 M$.
For each $k=1,\dots,p$, consider the frontal slices $\widehat{\mathcal{A}}_{\mathcal{E}}^{(k)}$. By \cite[Theorem 4.6]{hamm2021perturbations}, applying matrix CUR approximation to the frontal slices yields
\begin{equation*}
\|
\widehat{\mathcal{A}}^{(k)} - \widehat{\mathcal{C}}_{\mathcal{E}}^{(k)} (\widehat{\mathcal{U}}^{(k)}_{\mathcal{E}})^+\widehat{\mathcal{R}}^{(k)}_{\mathcal{E}}\|_{2}\leq c \| \widehat{\mathcal{E}}^{(k)}\|_{2},
\end{equation*}
where the constant $c$ depends only on $\|\widehat{\mathcal{C}}^{(k)}\|_2$, $\|(\widehat{\mathcal{U}}^{(k)})^+\|_2$ and $\|\widehat{\mathcal{R}}^{(k)}\|_2$.
By stacking the frontal slices and taking the maximum over $k$, we get 
\begin{equation*}
    \|\mathcal{A} - \widetilde{\mathcal{A}}\|_{2,*_M}\leq C \| \mathcal{E}\|_{2,*_M},
\end{equation*}
where $C$ depends only on $\|\mathcal{C}\|_{2,*_M}$, $\|\mathcal{U}^+\|_{2,*_M}$ and $\|\mathcal{R}\|_{2,*_M}$.
\end{proof}

\section{Numerical Experiments}\label{Sec:numerical_experiments}
\subsection{Formulation of the problem}

We apply the $*_M$-CUR tensor decomposition on video foreground-background separation.
In particular, given a tensor $\mathcal{X}$, the goal is to reconstruct a low-rank tensor $\mathcal{L^*}$ and a sparse tensor $\mathcal{S^*}$ such that
\begin{equation*}
    \mathcal{X}=\mathcal{L^*}+\mathcal{S^*}.
\end{equation*}
We remark that ``rank'' refers to the multirank in Definition \ref{defn:multirank}.
The non-convex optimization problem we thus aim to solve is
\begin{equation*}
    \min_{\mathcal{L},\mathcal{S}} \|\mathcal{A}-\mathcal{L}-\mathcal{S}\|_{F} \quad\text{for }\operatorname{rank}(\widehat{\mathcal{L}}^{(k)})\leq r \ \text{and sparse } \mathcal{S},
\end{equation*}
where  $\mathcal{L}$ and $\mathcal{S}$ are the outcomes of the reconstruction.
Specifically, for the $*_M$-CUR approximation of the low-rank component $\mathcal{L}$, we employ the Q-Discrete Empirical Interpolation Method (Q-DEIM) \citep{drmac2016new} for index selection.

\subsection{Video foreground-background separation}
We evaluate the $*_M$-CUR decomposition on different video sequences. 
To assess the performance of the $*_M$-CUR for different choices of matrix $M$, foreground–background separation is performed using the following matrices: the discrete cosine transform (DCT); discrete Fourier transform (DFT), where the $*_M$-product becomes the T-product; discrete sine transform (DST); and a data-dependent matrix $U_3$ obtained from the SVD of the mode-3 unfolding of the video sequence.
Table \ref{tab:reconstruction_metrics} reports background reconstruction metrics such as Average Gray-level Error (AGE), Percentage of Error Pixels (pEPs), and Peak Signal-to-Noise Ratio (PSNR), as well as the computational runtime for four different video sequences from the SBI dataset \cite{maddalena2015towards, bouwmans2017scene}. Given the background image ground truth, metrics are calculated for each frame and the average values are reported.
The runtime depends on the choice of $M$: in the injective case, $M\in\mathbb{R}^{2p\times p}$ yields the longest runtime; in the invertible case, $M\in\mathbb{R}^{p\times p}$; and in the surjective case, $M\in\mathbb{R}^{p\times 5}$ yields the lowest runtime.
Compared to the results reported in \cite{bouwmans2017scene}, the values in Table \ref{tab:reconstruction_metrics} show better performance for $*_M$-CUR in most metrics. We select a fixed $M$ for the subsequent experiments based on the values of Table \ref{tab:reconstruction_metrics}.

Figure \ref{fig:MCUR_comparison} showcases the difference in foreground reconstruction between the $*_M$-CUR, a matrix-wise robust CUR decomposition \citep{cai2021robust}, a tensor BM decomposition \citep{tian2025tensor}, and the SS-SVD \citep{kajo2018svd} for four video sequences of the CDnet dataset \citep{wang2014cdnet}. The foreground ground truth is included to allow visual comparison. 
From the results we observe that compared to the other CUR decomposition method, our proposed method shows less visual noise in some datasets (highway sequence) and slightly tighter boundaries of the foreground object in others (office sequence). In contrast, approximation models such as tensor BM and SS-SVD show vertical artifacts in the foreground reconstruction. We note also that the tensor BM decomposition took from 50--500 times longer than the other three methods, although the authors in \citep{tian2025tensor} discuss possible ways of reducing the runtime by using parallelization.

\begin{table*}[t]
\centering
\footnotesize
\setlength{\tabcolsep}{3pt}
\renewcommand{\arraystretch}{1.5}

\caption{Quantitative evaluation demonstrating the efficiency and accuracy of the $*_M$-CUR decomposition. We compare AGE, pEPs, PSNR, and runtime across different matrix types $M$. Entries are reported as \emph{surj / inv / inj}. The proposed method consistently achieves competitive performance, with the best values highlighted in bold.}
\label{tab:reconstruction_metrics}

\begin{tabular}{p{0.9cm}|l|r|r|r|r}
\toprule
Video & $M$ & AGE & pEPs & PSNR & Runtime (s) \\
\midrule

\multirow{4}{*}{\vseq{CAVIAR1}{[256,384,610]}}

& DCT &
12.4575 / 12.4689 / 12.4722 &
0.10918 / 0.11289 / 0.11329 &
24.85 / 24.84 / 24.83 &
\textbf{22.95} / 28.92 / 33.83 \\

& DFT &
12.2093 / \textbf{12.1580} / 12.2623 &
0.09992 / \textbf{0.09814} / 0.10099 &
25.01 / \textbf{25.05} / 24.98 &
25.42 / 43.29 / 62.00 \\

& DST &
12.3842 / 12.3870 / 12.4786 &
0.10828 / 0.11089 / 0.11229 &
24.89 / 24.88 / 24.84 &
24.31 / 27.87 / 36.35 \\

& U$_3$ &
12.3944 / 12.4056 / 12.4056 &
0.11064 / 0.11210 / 0.11210 &
24.85 / 24.86 / 24.86 &
27.94 / 34.18 / 39.53 \\

\midrule

\multirow{4}{*}{\vseq{HumanBody2}{[320,240,740]}}

& DCT &
4.6329 / 5.0693 / 4.9094 &
0.00608 / 0.00757 / 0.00676 &
32.09 / 31.43 / 31.70 &
18.55 / 25.32 / 31.15 \\

& DFT &
4.6986 / \textbf{4.5436} / 4.6058 &
\textbf{0.00582} / 0.00597 / 0.00591 &
\textbf{32.21} / 32.00 / 32.13 &
22.81 / 45.51 / 65.13 \\

& DST &
4.8380 / 5.1030 / 4.9303 &
0.00646 / 0.00766 / 0.00684 &
31.76 / 31.36 / 31.66 &
\textbf{18.49} / 23.95 / 28.61 \\

& U$_3$ &
5.5033 / 5.0501 / 5.0501 &
0.00776 / 0.00912 / 0.00912 &
30.88 / 31.31 / 31.31 &
23.34 / 30.67 / 35.91 \\

\midrule

\multirow{4}{*}{\vseq{IBMtest2}{[320,240,90]}}

& DCT &
3.4767 / 4.2296 / 3.9799 &
0.00101 / 0.00216 / 0.00193 &
35.21 / 33.51 / 33.93 &
\textbf{2.34} / 3.11 / 3.23 \\

& DFT &
\textbf{2.9598} / 3.0503 / 3.0102 &
\textbf{0.00082} / 0.00090 / \textbf{0.00082} &
\textbf{36.35} / 36.15 / 36.24 &
2.92 / 5.20 / 5.66 \\

& DST &
3.3238 / 3.9689 / 3.7869 &
0.00113 / 0.00179 / 0.00173 &
35.35 / 33.90 / 34.29 &
2.35 / 3.13 / 3.26 \\

& U$_3$ &
3.4828 / 4.1431 / 4.1431 &
0.00168 / 0.00228 / 0.00228 &
34.90 / 33.61 / 33.61 &
3.03 / 3.81 / 3.98 \\

\midrule

\multirow{4}{*}{\vseq{HighwayII}{[320,240,500]}}

& DCT &
3.7010 / 3.7292 / 3.7273 &
0.00479 / 0.00492 / 0.00486 &
32.11 / 32.06 / 32.08 &
12.98 / 15.55 / 18.10 \\

& DFT &
\textbf{3.6368} / 3.7286 / 3.6576 &
0.00471 / 0.00472 / \textbf{0.00468} &
\textbf{32.18} / 32.10 / 32.16 &
15.71 / 25.08 / 33.09 \\

& DST &
3.7043 / 3.7291 / 3.7004 &
0.00476 / 0.00488 / 0.00483 &
32.12 / 32.07 / 32.08 &
\textbf{12.52} / 15.57 / 17.97 \\

& U$_3$ &
3.6997 / 3.7641 / 3.7641 &
0.00481 / 0.00515 / 0.00515 &
32.16 / 32.05 / 32.05 &
15.62 / 18.63 / 21.14 \\

\bottomrule
\end{tabular}
\end{table*}

\begin{figure}[h!]
    \centering
    \includegraphics[width=0.95\linewidth]{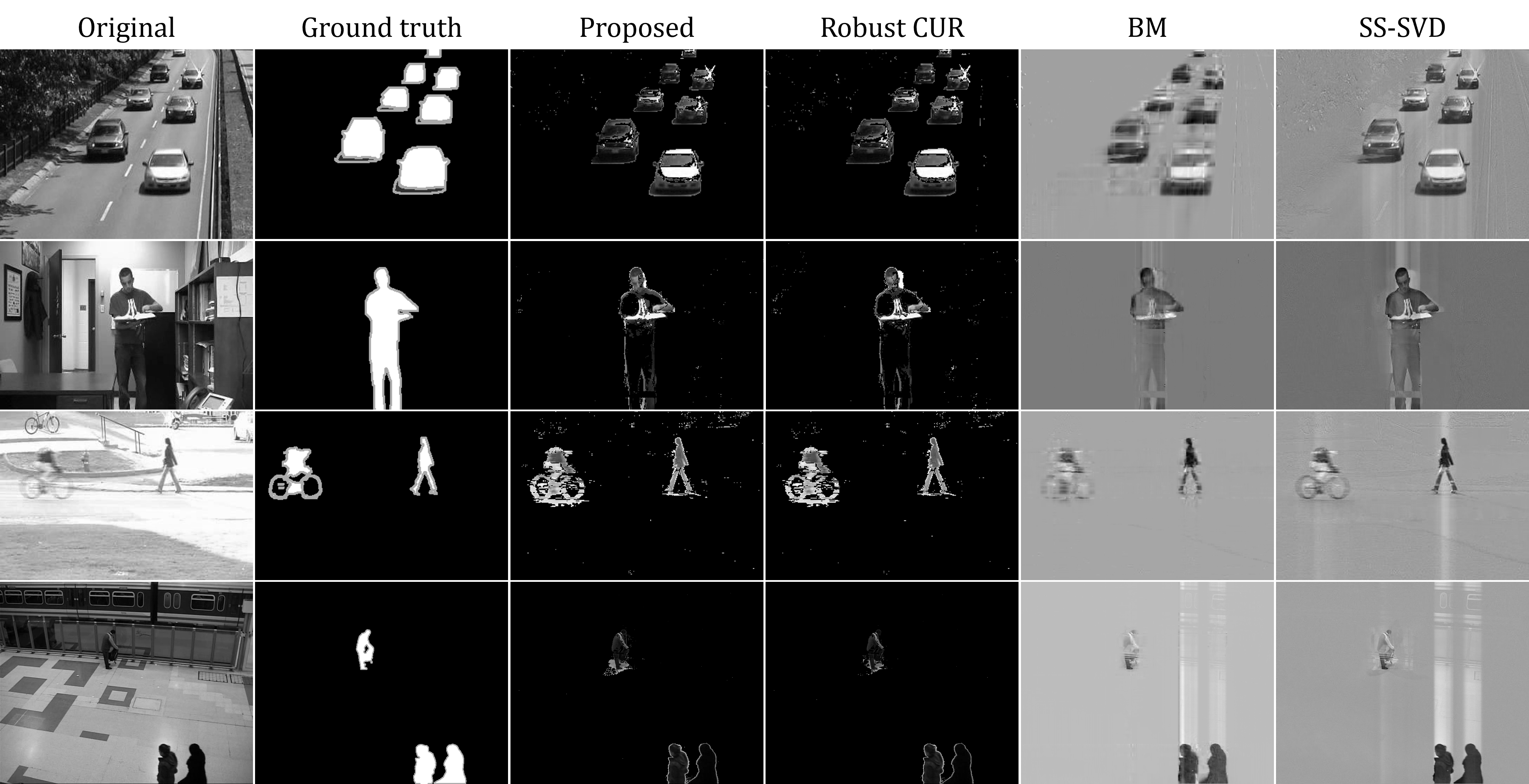}
    \caption{Visual comparison highlighting the robustness of the proposed $*_M$-CUR method against recent robust decomposition methods on the CDnet dataset. Our method (third column) effectively isolates the foreground with significantly reduced visual noise and sharper object boundaries compared to Robust CUR, Tensor BM, and SS-SVD. Selected frames: Highway (345), Office (1100), Pedestrians (175), and PETS2006 (620).}
    \label{fig:MCUR_comparison}
\end{figure}

\section{Conclusion}
We have introduced a tensor CUR decomposition under the $*_M$-product, which expands the family of tensor factorizations within the $*_M$-product algebraic setting.
Experiments on foreground-background separation on grayscale videos show that it can effectively perform the foreground-background separation task with compression and it visually outperforms the matrix-wise robust CUR decomposition in some datasets.

\section*{Data availability}
All experiments were implemented in Python and corresponding code and supplementary materials are available at the GitHub repository: 
\url{https://github.com/SusanaLM/LinearMapTensorCUR}.
Data sources are the CDnet dataset \citep{wang2014cdnet} available at \url{http://jacarini.dinf.usherbrooke.ca/dataset2014}, as well as the SBI dataset \cite{maddalena2015towards, bouwmans2017scene}, available at \url{https://sbmi2015.na.icar.cnr.it/SBIdataset.html}.

\section*{Acknowledgements}
The work of S. López-Moreno and J.-H. Lee was supported by the National Research Foundation of Korea (NRF) grant funded by the Korea government (MSIT) (RS-2024-00406152). Additionally, the work of S. López-Moreno was supported by Glocal University 30 Project at Pusan National University through the Institute for Regional System \& Education in Busan Metropolitan City, funded by the Ministry of Education (MOE) and the Busan Metropolitan City, Republic of Korea (2025-glocal-02-004-M432-01).
The work of T.~Kim was supported by the National Research Foundation of Korea (NRF) grant funded by the Korea government (MSIT) (No. 2022R1A5A1033624\&\,RS-2024-00342939\&\,RS-2025-25436769).
We would also like to thank Hyun-Min Kim and Jeong-Hoon Ju for their insightful comments.

{\footnotesize
\bibliographystyle{elsarticle-num-names} 
\bibliography{reference}
}

\end{document}